\documentclass{amsart}

\usepackage{graphicx}
\usepackage{color}

\newcommand{\ZZ}{\mathbb{Z}}

\newtheorem{theorem}{Theorem}[section]   
\newtheorem{lemma}[theorem]{Lemma}

\newtheorem{definition}[theorem]{Definition}

\newtheorem{remark}[theorem]{Remark}

\newtheorem{question}[theorem]{Question}

\begin{document}

\title[Non-hyperbolic automatic groups]{Non-hyperbolic automatic groups and groups acting on CAT(0) cube complex}

\author {Yoshiyuki Nakagawa}
\address{Department of Economics, Ryukoku University, Kyoto, 612-8577, Japan}
\email  {nakagawa@mail.ryukoku.ac.jp}

\author {Makoto Tamura}
\address{College of General Education,
         Osaka Sangyo University, Osaka, 574-8530, Japan}
\email  {mtamura@las.osaka-sandai.ac.jp}

\author {Yasushi Yamashita}

\address{Department of Information and Computer Sciences,
         Nara Women's University, Nara, 630-8506, Japan}
\email  {yamasita@ics.nara-wu.ac.jp}

\begin{abstract}
We discuss a problem posed by Gersten:
Is every automatic group which does not contain $\ZZ+\ZZ$ subgroup,
hyperbolic?
To study this question, we define the notion of ``$n$-tracks of length $n$'',
which is a structure like $\ZZ+\ZZ$,
and prove its existence in the non-hyperbolic automatic groups
with mild conditions.
As an application, we show that if a group
acts effectively, cellularly, properly discontinuously and
cocompactly on a CAT(0) cube complex and its quotient is "weakly special", 
then the above question is answered affirmatively.
\end{abstract}



\thanks{This work was supported by JSPS KAKENHI Grant Number 23540088}

\maketitle

\section{Introduction}
\label{sec:intro}

If a group $G$ has a finite $K(G, 1)$
and does not contain any Baumslag-Solitar groups, is $G$ hyperbolic?
(See \cite{Questions}.)
This is one of the most famous questions on hyperbolic groups.
Probably, many people expect that the answer is negative,
and it would be better to restrict our attention to some good class of groups.
In this paper we consider automatic groups.
If an automatic group $G$ does not contain any $\ZZ + \ZZ$ subgroups,
is $G$ hyperbolic?
Our problem is listed in \cite{openproblem} and attributed to Gersten.

Note that the class of all automatic groups contains
the class of all hyperbolic groups, all virtually abelian groups
and all geometrically finite hyperbolic groups \cite{MR1161694}.
A geometrically finite hyperbolic group is, in some sense,
similar to hyperbolic groups, but it might contain a $\ZZ+\ZZ$ subgroup.
Thus the class of automatic groups is a nice target
to consider the original question mentioned before.

Let us recall some related works very briefly.
If the group is the fundamental group of a closed $3$-manifold,
our question corresponds to the so-called ``weak hyperbolization''
of $3$-manifolds \cite{MR1362788}.
Also, D.~Wise proved that if the group satisfies
the small cancellation condition $B(6)$,
then the above question is answered affirmatively \cite{MR2053602}.
P.~Papasoglu proved that the Cayley graph of a group 
which is semihyperbolic but not hyperbolic
contains a subset quasi-isometric to $\ZZ + \ZZ$ \cite{MR1459144}. 

In this paper, we define the notion of ``$n$-tracks of length $n$'',
which suggests a clue of the existence of $\ZZ + \ZZ$ subgroup,
and show its existence in every non-hyperbolic automatic groups
with mild conditions.

As an application, we will show the next theorem:

{\bf Theorem\ \ref{thm:cube}\ } 
{\it Let $G$ be a group acting effectively, cellularly,
properly discontinuously and cocompactly
on a CAT(0) cube complex $X$.
If each hyperplane in $G\backslash X$ embeds
and does not self-osculate,
and $G$ is not word hyperbolic,
then $G$ contains $\ZZ + \ZZ$ as a subgroup.}

We remark that the assumption ``no self-osculating hyperplanes''
can be made weaker.
See section $4$ for the precise conditions we need.
See also Sageev and Wise \cite{MR2821442}.
They considered a similar problem for groups acting on CAT(0) cube-complex,
and introduced the notion of ``facing triple''.
We do not know the relation between our condition ``no direct self-contact''
in section 4
and theirs.

This paper is organized as follows.
In section $2$, we review definitions and some properties
of hyperbolic groups and automatic groups.
In section $3$, we introduce the notion of ``$n$-tracks''
and show its existence.
In section $4$, we review automatic structure of groups
which act effectively, cellularly, properly discontinuously
and cocompactly on CAT(0) cube complex due to
Niblo--Reeves~\cite{MR1604899},
and prove our theorem mentioned before.

\section{Hyperbolic groups and automatic groups}\label{sec:handa}

In this section, we briefly review definitions
and some properties of hyperbolic groups and automatic groups.
We refer to \cite{MR1161694} for the general theory.  

Let $G$ be a finitely generated group with a set of generators $A$.
In this paper, we will always assume that $A^{-1} = A$.
The \textit{Cayley graph} $\Gamma := \Gamma(G, A)$ of $G$
with respect to $A$ is a directed, labeled graph defined as follows:
the set of vertices is $G$ itself.
For $g, h \in G$, there is a directed edge $(g h)$,
source $g$ and target $h$, with label $a \in A$ if and only if $g a = h$.

Let $w$ be a word over $A$.
A \textit{prefix} of a word is any number of leading letters of that word.
We denote by $\ell(w)$ the word length of $w$
and by $w(t)$ the prefix of $w$ with length $t$.
The image of $w$ in $G$ by the natural projection is denoted by $\overline{w}$.
In this paper, we denote by $\overline{w(t_{1}, t_{2})}$
the subpath of the image of $w$ by the natural projection on $\Gamma$
from the vertex $\overline{w(t_{1})}$ to the vertex $\overline{w(t_{2})}$.
The Cayley graph $\Gamma$ is a metric space by its path metric.
We denote this metric by $d(g, h)$ for $g, h \in G$.

\begin{definition}
A geodesic space is said to be \textit{hyperbolic}
(in the sense of Gromov \cite{MR919829})
if there is a number $\delta > 0$ such that,
for any triangle $\triangle x y z$ with geodesic sides,
the distance from a point $u$ on one side to the union of the other two sides
is bounded by $\delta$.
A group $G$ with a set of generators $A$ is called word hyperbolic
if the Cayley graph $\Gamma$ is hyperbolic.
\end{definition}

It should be noted that the definition of a word hyperbolic group
does not depend on the choice of generators.
One of the most important properties of word hyperbolic groups
is the following theorem.

\begin{theorem}\label{thm:z2nothyperbolic}
If $G$ contains a $\ZZ + \ZZ$ subgroup, then $G$ can not be word hyperbolic.
(See \cite{MR1170363}.)
\end{theorem}

Next, we recall the concept of automatic structure.
Again, let $G$ be a finitely generated group with a set of generators $A$.
We denote by $\varepsilon$ the identity element of $G$.
A special letter $\$ \not\in A$ is used
to define the automatic structure of the group.
A finite state automaton $M$ over an alphabet $A$ is a machine
that determines ``accept'' or ``reject'' for a given word over $A$.
See \cite{MR1161694} for detail.
The language given by all the accepted words of a finite state automaton $M$
is denoted by $L(M)$.

\begin{definition}
\textit{An automatic structure} on $G$ consists of
a finite state automaton $W$ over $A$
and finite state automata $M_{x}$ over
$\left(A \cup \{\$\}\right) \times \left(A \cup \{\$\} \right)$,
for $x \in A \cup \{\varepsilon\}$, satisfying the following conditions:
\begin{enumerate}
\item The natural projection from $L(W)$ to $G$ is surjective.
\item For $x \in A \cup \{ \varepsilon \}$,
  we have $(w, w') \in L(M_x)$ if and only if
  $\overline{w \, x\mathstrut} = \overline{w'\mathstrut}$
  and both $w$ and $w'$ are elements of $L(W)$.
\end{enumerate}
For $M_{x}$ ($x \in A \cup \{\varepsilon\}$),
we think of the input $(w, w')$,
where $w = x_1 \, x_2 \cdots x_n$ and $w' = x_1' \, x_2' \cdots x_m'$
($x_i, x_j' \in A$, $i = 1, \ldots, n$ and $j= 1, \ldots, m$), 
as the string $(x_1, x_1') \, (x_2, x_2') \cdots$ defined over
$\left( A \cup \{ \$ \} \right) \times \left( A \cup \{ \$ \} \right)$.
If the word length of $w$ is not equal to $w'$,
say $\ell(w) < \ell(w')$,
we use the pudding letter $\$ $ and the input for the automaton is
$(x_1, x_1') \, (x_2, x_2') \cdots (x_n, x_n') (\$, x_{n+1}') \cdots (\$, x_m')$.
\end{definition}

$W$ is called the {\em word acceptor},
and each $M_{x}$ is called a \textit{compare automaton}
for the automatic structure.
An automatic group is one that admits an automatic structure.

\begin{lemma}[Lemma 2.3.2 of \cite{MR1161694}] \label{lem:kfellow}
If $G$ has an automatic structure,
there is a constant $k$ with the following property:
If $(w_1, w_2)$ is accepted by one of the automata
$M_x$, for $x \in A \cup \{\varepsilon\}$,
then $d(\overline{w_1(t)} ,\overline{w_2(t)}) < k$ 
for any integer $t\geq 0$.
\end{lemma}
Such a number $k$ is called a {\em fellow traveler's constant}
for the structure.

We define some properties of automatic structures which will be assumed later.

\begin{definition}
Let $W$ be the word accepter of an automatic structure of a group $G$
with generators $A$.
\begin{enumerate}
\item The automatic structure is \textit{prefix closed} if, for every
  $w \in L\left( W \right)$, any prefix $w\left( t \right) \, \left( 0
  \leq t \leq \ell\left( w \right) \right)$ is an element of
  $L\left( W \right)$.
\item The automatic structure has the \textit{uniqueness property} if
  the natural projection from $L\left( W \right)$ to $G$ is injective,
  thus bijective.
\item The group is \textit{weakly geodesically automatic} if any word
  $w \in L\left( W \right)$ is a geodesic with respect to path metric
  of $\Gamma$.
\item The group is \textit{strongly geodesically automatic} if
  $L\left( W \right)$ is equal to the set of all geodesic words.
\end{enumerate}
\end{definition}

In this paper, we investigate the relation between the word hyperbolicity
and the automaticity for finitely generated groups.
Here is the basic fact about the relation.

\begin{theorem}[Papasoglu \cite{MR1346209}]
Any finitely presented group is word hyperbolic
if and only if it is strongly geodesically automatic.
\end{theorem}

In the proof of the above theorem, the following lemma was proved.

\begin{lemma}\label{lem:mthickbigon}
If a group is not hyperbolic and weakly geodesically automatic, 
then for any large $M > 0$,
there exists a pair of geodesic words $(b_1, b_2)$  
such that $\overline{b_1 \mathstrut} = \overline{b_2 \mathstrut}$
and $d\left(\overline{b_1(r)}, \overline{b_2(r)}\right) > M$ for some $r$.
\end{lemma}

See Fig.~\ref{fig:mthick}.
We call $(b_1, b_2)$ in Fig.~\ref{fig:mthick}
\textit{$M$-thick bigon} with side $b_1, b_2$.

\section{Existence of $n$-tracks in non-hyperbolic automatic groups}
\label{sec:track}

Let $G$ be an automatic group with automatic structure
$(A, W, \{ M_x \}_{x \in A \cup \{ \varepsilon \}})$
where $A$ is the set of generators with $A^{-1} = A$,
$W$ the word acceptor
and $M_{x}$ the compare automaton for $x \in A \cup \{ \varepsilon \}$.
The following is the key concept in this paper.

\begin{definition}\label{def:track}
Let $T = \{t_1, t_2, \ldots, t_n\}$
be a set of mutually disjoint $n$ paths of length $n$ in $\Gamma$.  
We call $T$ $n$-tracks of length $n$
if there exist $2 \, n$ words
$w_1, w_1', w_2, w_2',$ $\ldots, w_n, w_n'$ of $L(W)$ 
and a positive integer $r$ such
that $(w_i', w_{i+1})$ is accepted by some compare automaton
for $i = 1, 2,\ldots , n-1$,
and that $t_i = \overline{w_i(r, r+n) \mathstrut}
= \overline{w_i'(r, r+n) \mathstrut}$ for $i=1, 2, \ldots, n$.
See Fig.~\ref{fig:deftrack}.
\end{definition}

\begin{figure}[ht]
  \centering

\begin{picture}(0,0)%
\includegraphics{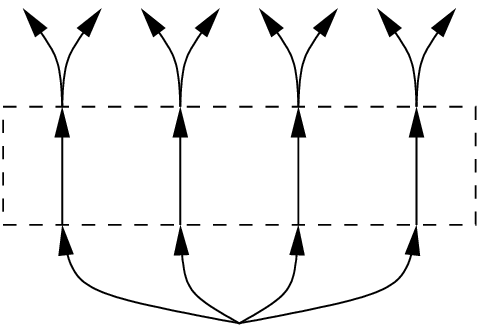}%
\end{picture}%
\setlength{\unitlength}{4144sp}%
\begingroup\makeatletter\ifx\SetFigFont\undefined%
\gdef\SetFigFont#1#2#3#4#5{%
  \reset@font\fontsize{#1}{#2pt}%
  \fontfamily{#3}\fontseries{#4}\fontshape{#5}%
  \selectfont}%
\fi\endgroup%
\begin{picture}(2187,1810)(256,-1151)
\put(271,524){\makebox(0,0)[lb]{\smash{{\SetFigFont{10}{12.0}{\familydefault}{\mddefault}{\updefault}{\color[rgb]{0,0,0}$w_1$}%
}}}}
\put(586,524){\makebox(0,0)[lb]{\smash{{\SetFigFont{10}{12.0}{\familydefault}{\mddefault}{\updefault}{\color[rgb]{0,0,0}$w_1'$}%
}}}}
\put(811,524){\makebox(0,0)[lb]{\smash{{\SetFigFont{10}{12.0}{\familydefault}{\mddefault}{\updefault}{\color[rgb]{0,0,0}$w_2$}%
}}}}
\put(1126,524){\makebox(0,0)[lb]{\smash{{\SetFigFont{10}{12.0}{\familydefault}{\mddefault}{\updefault}{\color[rgb]{0,0,0}$w_2'$}%
}}}}
\put(1351,524){\makebox(0,0)[lb]{\smash{{\SetFigFont{10}{12.0}{\familydefault}{\mddefault}{\updefault}{\color[rgb]{0,0,0}$w_3$}%
}}}}
\put(1666,524){\makebox(0,0)[lb]{\smash{{\SetFigFont{10}{12.0}{\familydefault}{\mddefault}{\updefault}{\color[rgb]{0,0,0}$w_3'$}%
}}}}
\put(1891,524){\makebox(0,0)[lb]{\smash{{\SetFigFont{10}{12.0}{\familydefault}{\mddefault}{\updefault}{\color[rgb]{0,0,0}$w_4$}%
}}}}
\put(2206,524){\makebox(0,0)[lb]{\smash{{\SetFigFont{10}{12.0}{\familydefault}{\mddefault}{\updefault}{\color[rgb]{0,0,0}$w_4'$}%
}}}}
\put(586,-331){\makebox(0,0)[lb]{\smash{{\SetFigFont{10}{12.0}{\familydefault}{\mddefault}{\updefault}{\color[rgb]{0,0,0}$t_1$}%
}}}}
\put(1126,-331){\makebox(0,0)[lb]{\smash{{\SetFigFont{10}{12.0}{\familydefault}{\mddefault}{\updefault}{\color[rgb]{0,0,0}$t_2$}%
}}}}
\put(1666,-331){\makebox(0,0)[lb]{\smash{{\SetFigFont{10}{12.0}{\familydefault}{\mddefault}{\updefault}{\color[rgb]{0,0,0}$t_3$}%
}}}}
\put(2206,-331){\makebox(0,0)[lb]{\smash{{\SetFigFont{10}{12.0}{\familydefault}{\mddefault}{\updefault}{\color[rgb]{0,0,0}$t_4$}%
}}}}
\put(1306,-1096){\makebox(0,0)[lb]{\smash{{\SetFigFont{10}{12.0}{\familydefault}{\mddefault}{\updefault}{\color[rgb]{0,0,0}$e$}%
}}}}
\end{picture}%

  \caption{$4$-track $T=\{t_1, t_2, t_3, t_4\}$ and its related paths}
  \label{fig:deftrack}
\end{figure}

The purpose of this section is to prove the following theorem.

\begin{theorem}\label{thm:track}
Let $G$ be a weakly geodesically automatic group
whose automatic structure $(A, W, \{ M_x \}_{x \in A \cup \{\varepsilon\}})$
is prefix closed and has the uniqueness property.
If $G$ is not hyperbolic,
then it contains $n$-tracks of length $n$ for any $n > 0$.
\end{theorem}

\begin{proof}
Let $k$ be a fellow traveler's constant for the automatic structure
and set $M = 2 \, k \, (n+1)^{2}$.
By Lemma~\ref{lem:mthickbigon}, there exists a $M$-thick bigon in $\Gamma$.
We denote by $b_1$ and $b_2$ the two sides of this $M$-thick bigon,
and by $e$ and $g$ the two end points.
Without loss of generality,
we may assume that $e$ is the identity vertex $\varepsilon$.

Since the automatic structure is weakly geodesically automatic,
there exists a word $p_0$ in $L(W)$
whose image $\overline{p_0}$ in $\Gamma$ is a geodesic from $e$ to $g$.
Then, at least one of two bigons $(p_0, b_1)$ and $(p_0, b_2)$ is $(M/2)$-thick.
We denote this bigon by $B = (p_0, b)$.
(See Fig.~\ref{fig:mthick})

\begin{figure}[htbp]
  \centering

\begin{picture}(0,0)%
\includegraphics{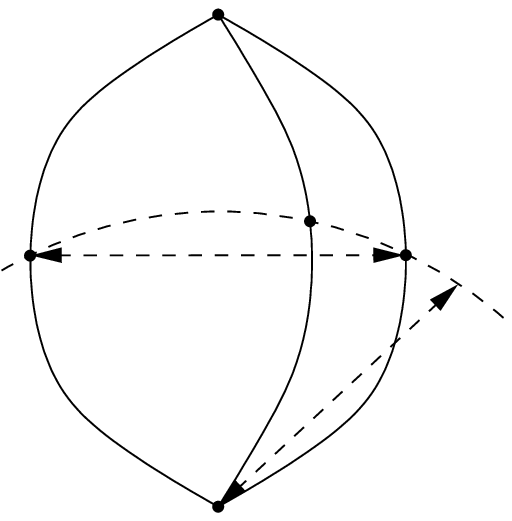}%
\end{picture}%
\setlength{\unitlength}{4144sp}%
\begingroup\makeatletter\ifx\SetFigFont\undefined%
\gdef\SetFigFont#1#2#3#4#5{%
  \reset@font\fontsize{#1}{#2pt}%
  \fontfamily{#3}\fontseries{#4}\fontshape{#5}%
  \selectfont}%
\fi\endgroup%
\begin{picture}(2314,2767)(353,-1646)
\put(1801,-48){\makebox(0,0)[lb]{\smash{{\SetFigFont{10}{12.0}{\familydefault}{\mddefault}{\updefault}{\color[rgb]{0,0,0}$p_0(r)$}%
}}}}
\put(451,-241){\makebox(0,0)[rb]{\smash{{\SetFigFont{10}{12.0}{\familydefault}{\mddefault}{\updefault}{\color[rgb]{0,0,0}$b_1(r)$}%
}}}}
\put(631,389){\makebox(0,0)[rb]{\smash{{\SetFigFont{10}{12.0}{\familydefault}{\mddefault}{\updefault}{\color[rgb]{0,0,0}$b_1$}%
}}}}
\put(2071,389){\makebox(0,0)[lb]{\smash{{\SetFigFont{10}{12.0}{\familydefault}{\mddefault}{\updefault}{\color[rgb]{0,0,0}$b_2$}%
}}}}
\put(1621,209){\makebox(0,0)[rb]{\smash{{\SetFigFont{10}{12.0}{\familydefault}{\mddefault}{\updefault}{\color[rgb]{0,0,0}$p_0$}%
}}}}
\put(1216,-601){\makebox(0,0)[lb]{\smash{{\SetFigFont{10}{12.0}{\familydefault}{\mddefault}{\updefault}{\color[rgb]{0,0,0}$>M$}%
}}}}
\put(2296,-241){\makebox(0,0)[lb]{\smash{{\SetFigFont{10}{12.0}{\familydefault}{\mddefault}{\updefault}{\color[rgb]{0,0,0}$b_2(r)$}%
}}}}
\put(2026,-736){\makebox(0,0)[rb]{\smash{{\SetFigFont{10}{12.0}{\familydefault}{\mddefault}{\updefault}{\color[rgb]{0,0,0}$r$}%
}}}}
\put(676,-421){\makebox(0,0)[lb]{\smash{{\SetFigFont{10}{12.0}{\familydefault}{\mddefault}{\updefault}{\color[rgb]{0,0,0}$d(b_1(r),b_2(r))$}%
}}}}
\put(1351,-1591){\makebox(0,0)[b]{\smash{{\SetFigFont{10}{12.0}{\familydefault}{\mddefault}{\updefault}{\color[rgb]{0,0,0}$e$}%
}}}}
\put(1351,974){\makebox(0,0)[b]{\smash{{\SetFigFont{10}{12.0}{\familydefault}{\mddefault}{\updefault}{\color[rgb]{0,0,0}$g$}%
}}}}
\end{picture}%

  \caption{$M$-thick bigon}
  \label{fig:mthick}
\end{figure}

By definition, we can find paths $p_i \in L(W)$ from $e$
to $\overline{b(\ell(b) - i)}$ for $i = 0, 1, 2, \ldots, \ell(b)$. 
Write $P = \{ p_i \}_{i = 0}^{\ell(b)}$.
Since the automatic structure is weakly geodesically automatic,
each $p_i \in P$ is geodesic
and $\ell(p_i) = \ell(b(\ell(b)-i)) = \ell(b) - i$. 

We claim that the intersection of two distinct paths $p_i$ and $p_j$
($i \neq j$) of $P$ is their common prefix
(possibly the identity vertex $e$) only.
To see this, suppose that $\overline{p_i}$ and $\overline{p_j}$ in $\Gamma$
have an intersection $\overline{p_i(t_i)} = \overline{p_j(t_j)}$ in $G$. 
Since the automatic structure is prefix closed,
both prefixes $p_i(t_i)$ of $p_i$ and $p_j(t_j)$ of $p_j$ are in $L(W)$.
Then, uniqueness property implies that $p_i(t_i) = p_j(t_j)$
(thus $t_i = t_j$) and the claim is proved.

Since $B = (p_0, b)$ is $(M/2)$-thick bigon,
there exists a number $r$ such that $d(p_0(r), b(r)) \geq M/2$.
Let $\Lambda_j$ \, ($j = 0,1,2, \ldots$) be a graph 
whose vertex set $V_j$ is the subset
$\left\{\overline{p_i(r+jn)} \, \mathstrut
\vert \, \ell(p_i) \geq r+jn \right\}$ of $V(\Gamma)$, 
and whose edge set $E_j$ is
$\left\{ \left( \overline{p_i(r+jn)},\, \overline{p_{i+1}(r+jn)} \right) \,
\mathstrut \vert \, i = 0, 1, 2, \ldots , \ell(b)-(r+jn)-1 \right\}$. 
Let $\gamma_j$ be a shortest simple path in $\Lambda_j$
from $\overline{p_0(r+jn)}$ to $\overline{p_{\ell(b)-(r+jn)}(r+jn)}$,   
and $\gamma_J$ be the longest path
in $\left\{ \gamma_0, \gamma_1, \gamma_2, \ldots \right\}$.
We set $R = r+Jn$.  

Let $\gamma_{J+1}^0, \gamma_{J+1}^1,  \gamma_{J+1}^2, \ldots ,
\gamma_{J+1}^{\ell(\gamma_{J+1})}$ be the geodesic paths in $L(W)$ 
from $e$ to the vertices of $\gamma_{J+1}$. 
Note that $\gamma_{J+1}^0=p_0(R+n)$
and $\gamma_{J+1}^{\ell(\gamma_{J+1})}=p_{\ell(b)-R+n}(R+n)$. 
Let $t^i = \overline{\gamma_{J+1}^i(R, R+n)}$. 
By construction, $t^i$ and $t^{i+1}$ are subpaths of $p_m$ and $p_{m+1}$
respectively for some $m$, 
and $(p_m, p_{m+1})$ is accepted by $\mathcal M$. 
Let $T' = \left\{t^0, t^1, \ldots , t^{\ell(\gamma_{J+1})} \right\}$.
If consecutive $n$ paths in $T'$ are mutually disjoint, 
then they give $n$-tracks of length $n$. 
Note that if some of $T'$ intersect each other,
this gives branches in the union of $T'$ by the above claim.

Let $y$ be the number of branches in the union of the paths
$t^0, t^1, \ldots , t^{\ell(\gamma_{J+1})}$. 
(See thick curves in Fig.~\ref{fig:ntrack2}.)

\begin{figure}[htbp]
  \centering

\begin{picture}(0,0)%
\includegraphics{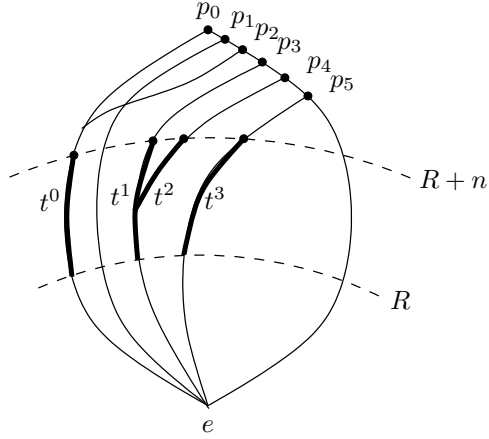}%
\end{picture}%
\setlength{\unitlength}{4144sp}%
\begingroup\makeatletter\ifx\SetFigFont\undefined%
\gdef\SetFigFont#1#2#3#4#5{%
  \reset@font\fontsize{#1}{#2pt}%
  \fontfamily{#3}\fontseries{#4}\fontshape{#5}%
  \selectfont}%
\fi\endgroup%
\begin{picture}(2471,2696)(156,-1626)
\put(2612,-109){\makebox(0,0)[lb]{\smash{{\SetFigFont{10}{12.0}{\familydefault}{\mddefault}{\updefault}{\color[rgb]{0,0,0}$R+n$}%
}}}}
\put(2439,-829){\makebox(0,0)[lb]{\smash{{\SetFigFont{10}{12.0}{\familydefault}{\mddefault}{\updefault}{\color[rgb]{0,0,0}$R$}%
}}}}
\put(1351,-1571){\makebox(0,0)[b]{\smash{{\SetFigFont{10}{12.0}{\familydefault}{\mddefault}{\updefault}{\color[rgb]{0,0,0}$e$}%
}}}}
\put(895,-211){\makebox(0,0)[rb]{\smash{{\SetFigFont{10}{12.0}{\familydefault}{\mddefault}{\updefault}{\color[rgb]{0,0,0}$t^1$}%
}}}}
\put(1026,-210){\makebox(0,0)[lb]{\smash{{\SetFigFont{10}{12.0}{\familydefault}{\mddefault}{\updefault}{\color[rgb]{0,0,0}$t^2$}%
}}}}
\put(1329,-234){\makebox(0,0)[lb]{\smash{{\SetFigFont{10}{12.0}{\familydefault}{\mddefault}{\updefault}{\color[rgb]{0,0,0}$t^3$}%
}}}}
\put(470,-246){\makebox(0,0)[rb]{\smash{{\SetFigFont{10}{12.0}{\familydefault}{\mddefault}{\updefault}{\color[rgb]{0,0,0}$t^0$}%
}}}}
\put(1357,935){\makebox(0,0)[b]{\smash{{\SetFigFont{10}{12.0}{\familydefault}{\mddefault}{\updefault}{\color[rgb]{0,0,0}$p_0$}%
}}}}
\put(1564,876){\makebox(0,0)[b]{\smash{{\SetFigFont{10}{12.0}{\familydefault}{\mddefault}{\updefault}{\color[rgb]{0,0,0}$p_1$}%
}}}}
\put(1706,806){\makebox(0,0)[b]{\smash{{\SetFigFont{10}{12.0}{\familydefault}{\mddefault}{\updefault}{\color[rgb]{0,0,0}$p_2$}%
}}}}
\put(1841,726){\makebox(0,0)[b]{\smash{{\SetFigFont{10}{12.0}{\familydefault}{\mddefault}{\updefault}{\color[rgb]{0,0,0}$p_3$}%
}}}}
\put(2020,613){\makebox(0,0)[b]{\smash{{\SetFigFont{10}{12.0}{\familydefault}{\mddefault}{\updefault}{\color[rgb]{0,0,0}$p_4$}%
}}}}
\put(2156,499){\makebox(0,0)[b]{\smash{{\SetFigFont{10}{12.0}{\familydefault}{\mddefault}{\updefault}{\color[rgb]{0,0,0}$p_5$}%
}}}}
\end{picture}%

  \caption{Finding $n$-track of length $n$.}
  \label{fig:ntrack2}
\end{figure}

We claim that $y \leq n$.
To see this, let $\widehat{\gamma_{J+1}}$ be the image of $\gamma_{J+1}$
by the natural projection $\pi$ from $\Lambda_{J+1}$ to $\Lambda_{J}$, 
that is, $\pi(\overline{\gamma_{J+1}^i(R+n)}) = \overline{\gamma_{J+1}^i(R)}$. 
We have $\ell(\widehat{\gamma_{J+1}}) \leq \ell(\gamma_{J+1}) - y$. 
Since there are $n$ endpoints $\overline{p_{\ell(b)-R}(R)}$, 
$\overline{p_{\ell(b)-(R+1)}(R+1)}$, $\ldots$,
$\overline{p_{\ell(b)-(R+n-1)}(R+n-1)}$ of $p_i$'s
between $\Lambda_J$ and $\Lambda_{J+1}$,
and $\gamma_{J}$ is the shortest in $\Lambda_{J}$, 
we have $\ell(\gamma_{J}) \leq \ell(\widehat{\gamma_{J+1}}) + n$. 
Since $\gamma_J$ is the longest in
$\left\{ \gamma_0, \gamma_1, \gamma_2, \ldots \right\}$,
it follows that $\ell({\gamma_{J+1}}) \leq \ell(\gamma_J)$. 
Therefore 
$\ell(\widehat{\gamma_{J+1}}) + y \leq \ell(\gamma_{J+1})
\leq \ell(\gamma_{J}) \leq \ell(\widehat{\gamma_{J+1}}) + n$, 
thus we have $y \leq n$ and the claim is proved.

Recall that we set $M = 2k(n+1)^2$ at the beginning of this proof,
where $k$ is a fellow traveler's constant.
Hence we have $\ell(\gamma_J) \geq \ell(\gamma_0) > M/2k = (n+1)^2$.
Since there are at most $n$ branches in 
$T'$ and the number of elements in $T'$ is $(n+1)^2$,
there exist consecutive $n$ paths in $T'$ with the desired property,
and the theorem is proved. 
\end{proof}

\section{CAT(0) cube complexes}
\label{sec:cube}

Does the existence of $n$-track of length $n$ for any $n$
imply the existence of $\ZZ + \ZZ$ subgroup?
We do not have the complete answer.
But, as an application of the theorem in the previous section,
we give a partial answer to this question
for the groups acting on CAT(0) cube complexes.

See \cite{MR2605177}, for CAT(0) and its relation to hyperbolicity.
See also \cite{MR2949207}, \cite{MR2821442}.

\subsection{The automatic structure for groups acting on CAT(0) cube complex}
\label{subsec:NR}

In this subsection, we briefly review 
the automatic structures given by 
Niblo and Reeves \cite{MR1604899}.

An $n$-cube is a copy of $[-1,1]^n$.
A cube complex is obtained from a collection of cubes
of various dimensions by identifying certain subcubes.
A flag complex is a simplicial complex
with the property that every finite set of pairwise
adjacent vertices spans a simplex.
Let $X$ be a cube complex.
The link of a vertex $v$ in $X$ is a complex built from simplices
corresponding to the corners of cubes adjacent to $v$.

\begin{definition}
A cube complex $X$ is nonpositively curved
if, for each vertex $v$ in $X$, $\text{link}(v)$ is a flag complex.
\end{definition}

Gromov \cite{MR919829} showed that a cube complex is CAT(0)
if and only if it is simply connected and nonpositively curved.
Many groups studied in combinatorial group theory act
properly and cocompactly on CAT(0) cube complexes.

Let us recall the definition of hyperplane for cube complex.
Our reference here is \cite{MR2377497}.
See also \cite{MR2979855}.
A midplane in a cube $[-1,1]^n$ is the subspace
obtained by restricting exactly one coordinate to $0$.
Given an edge in a cube, there is a unique midplane
which cuts the edge transversely.
A hyperplane $H$ of a cube complex $X$
is obtained by developing the midplanes in $X$,
i.e., identifying common subcubes of midplanes which cuts the same edge.
These edges are said to be dual to $H$.

This is a basic fact about hyperplane.

\begin{lemma}[Proposition 2.7 in \cite{MR1604899}]\label{lem:sep}
Every hyperplane in CAT(0) cube complex $X$ separates $X$ into exactly two components.
\end{lemma}

Each component is referred to as the halfspace associated with $H$.

Let $X$ be a CAT(0) cube complex and
consider a sequence of cubes $\{C_i\}_0^n$ in $X$,
each of dimension at least 1,
such that each cube meets its successor in a single vertex
$\tilde v_i = C_{i-1} \cap C_i$.
This sequence is called a cube-path
if $C_i$ is the cube of minimal dimension containing $\tilde v_i$ and $\tilde v_{i+1}$.
Let $\tilde v_0$ to be the vertex of $C_0$, which is diagonally opposite to $\tilde v_l$,
and $\tilde v_{n+1}$ the vertex of $C_n$, diagonally opposite to $\tilde v_n$.
$\tilde v_0$ is called the initial vertex and $\tilde v_{n+1}$ the terminal vertex.
For a cube $C\in X$, $St(C)$ is the union of all cubes
which contain $C$ as a subface (including C itself).

\begin{definition}[Definition 3.1 in \cite{MR1604899}]
A cube-path is called a normal cube-path if $C_i \cap St(C_{i-1}) = \tilde v_i$.
\end{definition}

\begin{lemma}\label{lem:unqp}
Given two vertices $\iota, \tau\in V(X)$, 
there is a unique normal cube-path from $\iota$ to $\tau$.
(Proposition 3.3 in \cite{MR1604899}).
A normal cube-path achieves the minimum length among
all cube-paths joining the endpoints 
(See remark in section 3 in \cite{MR1604899} and \cite{ReevesPHD}.)
\end{lemma}

\begin{remark}[Remark at the end of section 3 in \cite{MR1604899}]\label{rem:NR3}
Given a vertex $\tilde v$ on a normal cube-path, which terminates at $\tau$,
the cube following $\tilde v$ is spanned by the planes
which meet $St(\tilde v)$ and separate $v$ from $\tau$.
\end{remark}

Let $X$ be a CAT(0) cube complex, and $V(X)$ its vertex set.
Let $G$ be a group acting effectively, cellularly, properly
discontinuously and cocompactly on $X$.
Let $G\backslash X$ denote the quotient
of the complex $X$ by the action of $G$.
The fundamental groupoid $\pi(G\backslash X)$ is the groupoid
whose objects are the points of $G\backslash X$ and morphisms
between points $v, v'$ are homotopy classes of paths in $G\backslash X$
beginning at $v$ and ending at $v'$.
The multiplication in $\pi(G\backslash X)$
is induced by composition of paths.

A directed cube is a cube 
with two ordered diagonally opposite vertices specified.
Let $A$ be the set of 
homotopy classes of the diagonal of all directed cubes in $G\backslash X$.
The correspondence between $A$ and directed cubes in $G\backslash X$
is one to one.
The directed cubes in $X$ can be labelled equivariantly by 
(the lifts of) $A$,
so each cube-path in X defines a word in $A^*$.
Let $\mathcal L$ be the subset of $A^*$
which corresponds to normal cube-paths.

\begin{lemma}\label{lem:as}
Let $A$ and $\mathcal L$ be as above.  Then we have:
\begin{enumerate}
\item There exists an isometry between $\pi(G\backslash X)$ with the 
  word metric given by $A$ and $V(X)$ with the metric given by
  normal cube paths.
  (Lemma 4.1 in \cite{MR1604899})
\item $\mathcal L$ is regular over $A$.
  (Proposition 5.1 in \cite{MR1604899})
\item $\mathcal L$ satisfies 1-fellow travel property. 
  (Proposition 5.2 in \cite{MR1604899})
\end{enumerate}
In particular, $({A, L})$ induces an automatic structure for 
$\pi(G\backslash X)$. (See Theorem 5.3 in \cite{MR1604899})
This structure is prefix closed, weakly geodesically automatic
with uniqueness property. ( (1)  and Lemma~\ref{lem:unqp} )
\end{lemma}

The set of states of (non-deterministic) finite-state automaton
for $\mathcal L$ is $A$. (Proposition 5.1 in \cite{MR1604899})
Thus, There is a natural map from the set of states of the word acceptor
of $\pi(G\backslash X)$ to $G\backslash X$
by taking the tail of directed cubes.

In this section, for vertices $\tilde{v}, \tilde{u}$ in $X$, 
we denote by $d(\tilde{v},\tilde{u})$ the distance
given by normal-cube paths.

Let $v$ be a vertex in $G\backslash X$.
The group $G$ is realized as a subgroupoid $\pi(G\backslash X, \{v\})$
whose object is $v$ only,
and whose morphisms are all the morphisms of $\pi(G\backslash X)$
between $v$.
It is easy to construct an automatic structure
for the group $G=\pi(G\backslash X, \{v\})$
from the automatic structure for the groupoid $\pi(G\backslash X)$.

\subsection{Standard automata}

Let $G$ be a group or groupoid with automatic structure $M =
\left( A, W, \left\{ M_{x} \right\}_{x \in A \cup \left\{ \varepsilon \right\}} \right)$,
and $k$ a fellow traveler's constant for $M$.
For later purpose, we construct an automaton $\mathcal M$ from $M$.
It is called standard automata in \cite{MR1161694}
when $G$ is a group.
(See Definition 2.3.3 in \cite{MR1161694}.)
Put
\[
S':= \left\{ ( s,t,g ) \; \vert \;
     s, t \in S_{W}, \: s, t \neq F_{W}, \: \ell(g) \leq k \right\}
\]
where $S_{W}$ is the state set of $W$ and $F_{W}$ is the failure state of $W$.
The state set $S$ of $\mathcal M$ is $S' \cup \{ \mbox{failure state } F \}$.
The initial state of $\mathcal M$ is $\left( s_{0}, s_0, id \right)$,
where $s_{0}$ is the initial state of $W$.
The transition function $\mu$ of $\mathcal M$ is:
\[
 \mu((s,t,g), (x,y)) =
\begin{cases}
\left(\mu_W(s,x),\: \mu_W(t,y),\: x^{-1}\,g\,y\right)\;
         & \mbox{ if it is in } S',\\
      F  & \mbox{ otherwise,}
\end{cases}
\]
where $\mu_{W}$ is the transition function of $W$.
Note that $\mu$ can be extended in a unique way
to a map $A^{\ast} \times A^{\ast} \to S$ and we also denote it by $\mu$.

\subsection{Groups acting on CAT(0) cube complexes}

Let $G$ be a group acting effectively,
cellularly, properly discontinuously and cocompactly 
on a CAT(0) cube complex $X$.
Let $\tilde v$ be a vertex and $H$ a hyperplane in $X$.
Let ${\mathcal H}^-$ be the halfspace associated with $H$ that does not contain $\tilde v$.
We define the distance between $\tilde v$ and $H$ as $d(\tilde v, H) = d+\frac12$,
where $d = \min \{d(\tilde v, \tilde v') | \tilde v' \in {\mathcal H}^-\} - 1$.
We denote by $N(H) \cong H\times [-1,1]$ the cubical neighborhood of $H$.

\begin{lemma}\label{lem:separate}
If $d(\tilde v, H) \geq \frac32$, there exists a hyperplane $H'$
such that $d(\tilde v,H')=\frac12$ and $H'$ separates $\tilde v$ and $H$.
\end{lemma}

We remark that $d(\tilde v,H')=\frac12$ if and only if $St(\tilde v)\cap H'\neq\emptyset$.

\begin{proof}
We denote the halfspace associated with $H$ that contains $\tilde v$ by ${\mathcal H}^+$.

We prove the lemma by induction on $d(\tilde v, H)$.
Suppose that $d(\tilde v, H) = \frac32$.
Let $\tilde v'$ be a vertex in ${{\mathcal H}^+} \cap N(H)$
such that $d(\tilde v, \tilde v') = 1$.
Let $C_1$ be the cube spanned by $\tilde v'$ and $\tilde v$. 
There exists an edge $e$ in $C_1$ adjacent to $\tilde v'$,  not contained in $N(H)$
such that the hyperplane $H'$ defined by $e$ (i.e., dual to $e$) separates $\tilde v'$ and $\tilde v$.
Then, $H'$ does not intersect $H$,
(See, for example, the proof of Lemma 2.14 in \cite{MR1604899}.)
and $H'$ separates $\tilde v$ and $H$.

Next, suppose that $d+\frac12 = d(\tilde v, H) > \frac32$.
Let $\tilde v'$ be a vertex in ${\mathcal H}^+ \cap N(H)$
such that $d(\tilde v, \tilde v') = d$.
Let $C_1, \dots, C_d$ be the normal cube-path from $\tilde v'$ to $\tilde v$.
Denote the vertex $C_1\cap C_2$ by $\tilde v_1$.
As in the base case, there exists a hyperplane $H''$ 
that separates $\tilde v_1$ and $\tilde v'$ and 
$H''$ does not intersect $H$.
Since $C_1, \dots, C_d$ is a normal cube-path, $H''$ separates $\tilde v$ and $H$,
and $d(\tilde v, H'') < d(\tilde v, H)$.
By induction, there exists a hyperplane $H'$ such that $d(\tilde v, H')=\frac12$
and $H'$ separates $\tilde v$ and $H''$. 
Clearly, $H'$ separates $\tilde v$ and $H$ and the lemma is proved.
\end{proof}

Next lemma is our key technical lemma.

\begin{lemma}\label{lem:distance}
Let $C_0, \dots, C_n$ be a normal cube-path, and 
$\tilde v_0, \dots, \tilde v_{n+1}$ the vertices of this cube path,
that is, $\tilde v_i=C_{i-1}\cap C_i$ for $i=1, \dots, n$
with $\tilde v_0$ the initial vertex and $\tilde v_{n+1}$ the terminal vertex.
Let $H$ be a hyperplane separating $\tilde v_0$ and $\tilde v_n$.
If $d(\tilde v_0, H)=d+\frac12$, then $H$ separates $\tilde v_d$ and $\tilde v_{d+1}$.
\end{lemma}

\begin{proof}
We prove the lemma by induction on $d$.

The base case $d=0$ is trivial 
as $d(\tilde v_0,H)=\frac12$ and $H$ meets $St(\tilde v)$.
Then, $H$ separates $\tilde v_0$ and $\tilde v_1$ by Remark~\ref{rem:NR3}.

Now, consider the case $d>0$.
Assume that $d(\tilde v_0,H)=d+\frac12$ and $H$ does not separate $\tilde v_d$ and $\tilde v_{d+1}$.
Since $d(\tilde v_0,H)=d+\frac12$,
$H$ can not separate $\tilde v_0$ and $\tilde v_d$.
Hence, $H$ separates $\tilde v_{d+1}$ and $\tilde v_n$.
By Remark~\ref{rem:NR3}, we see that $H\cap St(\tilde v_d)=\emptyset$.
By Lemma~\ref{lem:separate}, there exists a hyperplane $H'$
such that $H'\cap St(\tilde v_d) \neq \emptyset$ and 
$H'$ separates $\tilde v_d$ and $H$.
Then, $H'$ separates $\tilde v_d$ and $\tilde v_n$, and it follows that 
$H'$ separates $\tilde v_d$ and $\tilde v_{d+1}$. 
(See Remark~\ref{rem:NR3} again.)
By the induction hypothesis, $d(\tilde v_0, H') \geq d+\frac12$,
so we have $d(\tilde v_0, H')=d+\frac12.$
It follows that $d(\tilde v_0, H) > d(\tilde v_0, H') = d+\frac12$ and this is a contradiction.
\end{proof}

Let $\mathcal M$ be the standard automaton for the 
automatic structure of the groupoid $\pi(G\backslash X)$
given in~\ref{subsec:NR}.
We use the same symbols as in the previous subsection.
Let $(s,t,g)$ be a state in $\mathcal M$.
Since $\mathcal L$ (the set of words corresponding to normal cube-paths)
 satisfies 1-fellow travel property,
$g$ is in $A$ (the set of generators).
Recall that $A$ consists of directed cubes in $G\backslash X$.
We define the dimension the the state $(s,t,g)$,
denoted by $\dim(s,t,g)$, as the dimension of $g$ as a (directed) cube.
We also define $\dim(\text{failure state }F) = +\infty.$

\begin{lemma}\label{lem:mono}
For any transition $(s',t',g') = \mu((s,t,g), (x,y))$ in $\mathcal M$,
with $x\neq e$ and $y\neq e$,
we have $\dim(s',t',g') \geq \dim(s,t,g)$.
\end{lemma}

\begin{proof}
Fix $\iota\in V(X)$ as the base point of $X$.
Let $\tau_0, \tau_1$ be two points in $X$ such that
$d(\tau_0, \tau_1)=1$ and $d(\iota, \tau_0) = d(\iota, \tau_1)$.
Denote the associated vertices by 
$\iota=\tilde v_0, \ldots, \tilde v_n=\tau_0$ and 
$\iota=\tilde u_0, \ldots, \tilde u_n=\tau_1$.

Let $H$ be a hyperplane and put $d=d(\iota, H)$.
If $H$ separates $\tau_0$ and $\tau_1$, 
then, by Remark~\ref{rem:NR3} and Lemma~\ref{lem:distance}, 
$H$ does not separate $\tilde v_i$ and $\tilde u_i$ for $i=0, \dots, d$, 
and separates $\tilde v_i$ and $\tilde u_i$ for $i=d+1, \dots, n$.
If $H$ does not separate $\tau_0$ and $\tau_1$, 
but separate $\iota$ from both $\tau_0$, $\tau_1$,
then $H$ does not separate $\tilde v_i$ and $\tilde u_i$ for all $i=0, \dots, n$,
since $H$ separates $\tilde v_d$ and $\tilde v_{d+1}$ as well as $\tilde u_d$ and $\tilde u_{d+1}$.

It follows that the dimension of the cube spanned by $\tilde v_i$ and $\tilde u_i$ is 
equal to the number of hyperplanes separating $\tau_0$ and $\tau_1$
having distance to $\iota$ less than $i$.
Since this dimension is equal to the corresponding state in $\mathcal M$,
the dimensions of the states are 
monotone increasing under the transitions in $\mathcal M$.
\end{proof}

For vertices $\tilde v, \tilde v'$ in $X$, we denote the set of hyperplanes
separating $\tilde v$ and $\tilde v'$  by $S(\tilde v, \tilde v')$.
When $d(\tilde v, \tilde v')=1$, we denote by $[\tilde v, \tilde v'] (\in A)$
the label of the directed cube from $\tilde v$ to $\tilde v'$.
Since our automatic structure has uniqueness property,
there is a natural map $P: X \to S_W$,
where $S_W$ is the state set of the word acceptor for $\pi_1(G\backslash X)$.
We say that vertices $\tilde v, \tilde u, \tilde v', \tilde u'$ in $X$ 
corresponds to  a transition $(s',t',g') = \mu((s,t,g), (x,y))$ in $\mathcal M$
if $P(\tilde v)=s, P(\tilde u)=t, P(\tilde v')=s', P(\tilde u')=t',
[v,u]=g, [v,v']=x, [u,u']=y$ and $[v',u']=g'$.

\begin{lemma}\label{lem:spancube}
Suppose that the vertices $\tilde v, \tilde u, \tilde v', \tilde u'$ in $X$ 
corresponds to  a transition $(s',t',g') = \mu((s,t,g), (x,y))$ in $\mathcal M$
with $\dim(s',t',g') = \dim(s,t,g)$.
Then, $\tilde v, \tilde u, \tilde v', \tilde u'$ span a cube in $X$
such that
$S(\tilde v, \tilde v') = S(\tilde u, \tilde u')$ and
$S(\tilde v, \tilde u) = S(\tilde v', \tilde u')$
\end{lemma}

\begin{proof}
By the argument in the proof of Lemma~\ref{lem:mono},
it is clear that $S(\tilde v, \tilde u) = S(\tilde v', \tilde u')$, and
$S(\tilde v, \tilde v') = S(\tilde u, \tilde u')$.
Since $S(\tilde v, \tilde u') = S(\tilde v, \tilde u) \cup S(\tilde u, \tilde u')$,
we have $H\cap St(\tilde v)\neq\emptyset$ for each $H\in S(\tilde v, \tilde u')$.
By Lemma 2.14 in  \cite{MR1604899}, there exists a cube that this union spans,
and the lemma is proved.
\end{proof}

\begin{lemma}\label{lem:xy}
Consider two transitions
$(s',t',g') = \mu((s,t,g), (x,y))$ and
$(s'',t'',g'') = \mu((s,t,g), (x',y'))$ in $\mathcal M$,
and suppose that
$\dim(s,t,g) = \dim(s',t',g') = \dim(s'', t'', g'')$.
Then, $x=x'$ implies $y=y'$.
Similarly, $y=y'$ implies $x=x'$.
\end{lemma}

Note that the condition $\dim(s,t,g) = \dim(s',t',g') = \dim(s'', t'', g'')$ 
implies that both $(s',t',g')$ and $(s'',t'',g'')$ are not failure states,
whose dimension was defined as $+\infty$.

\begin{proof}
We will prove that $x=x'$ implies $y=y'$.
Note that we have $s' = s''$ in this case.

Suppose that vertices $\tilde v, \tilde u, \tilde v', \tilde u'$ in $X$ 
correspond to a transition $(s',t',g') = \mu((s,t,g), (x,y))$,
and vertices $\tilde v, \tilde u, \tilde v', \tilde u''$ in $X$ 
correspond to a transition $(s',t'',g') = \mu((s,t,g), (x',y'))$ in $\mathcal M$.
Let $C_1, C_2, C_3$, and $C_4$ be the cubes spanned by
$\{\tilde v, \tilde v'\}$, $\{\tilde v, \tilde u\}$,
 $\{\tilde v, \tilde v', \tilde u, \tilde u'\}$,
 and
$\{\tilde v, \tilde v', \tilde u, \tilde u''\}$,
respectively.
Then, in $link(v)$, the simplices corresponding to $C_3$ and $C_4$
are spanned by the simplices corresponding to $C_1$ and $C_2$.
Recall that $link(v)$ was a flag complex, because $X$ is nonpositively curved.
Thus, we have $C_3 = C_4$.
Hence, $u'=u''$ and we have $y = y'$. 
\end{proof}

To prove the main theorem of this section,
we need a stronger conclusion than the previous lemma.
In order to state the next lemma,
let us introduce some notation.  (See \cite{MR2377497} for more details.)

Let $\vec{a}, \vec{b}$ be oriented edges
having a common initial (or terminal) vertex $v$.
Oriented edges $\vec{a}$ and $\vec{b}$ are said to directly osculate at $v$
if they are not adjacent in $\text{link}(v)$.
Let $a, b$ be (unoriented) edges
having a common end point $v$.
Edges $a$ and $b$ are said to osculate at $v$
if they are not adjacent in $\text{link}(v)$.

We consider  hyperplanes in $G\backslash X$.
From now on, we assume that each hyperplane in $G\backslash X$ is embedding.

A hyperplane $H$ is said to be 2-sided
if its open cubical neighborhood is isomorphic to the product $H\times (-1,1)$.
If a hyperplane is not 2-sided, then it is said to be 1-sided.
If $H$ is 2-sided, one can orient dual edges in a consistent way.
A 2-sided hyperplane is said to directly self-osculate
if it is dual to distinct oriented edges that directly-osculate.
We say that 1-sided hyperplane self-osculates
if it is dual to distinct (unoriented) edges that osculate.

In this paper, we introduce the following notion:

\begin{definition}
We say that a 2-sided hyperplane $H$ self-contacts
if there are two vertices $u, v$ such that
$d(u,v)=1$ and $H$ directly self-osculates at 
$u$ and $v$.
(See Figure~\ref{fig:ws}.)
We say that a 1-sided hyperplane $H$ self-contacts
if there are two vertices $u, v$ such that
$d(u,v)=1$ and $H$ self-osculates at $u$ and $v$.
\end{definition}

\begin{figure}[htbp]
  \centering

\begin{picture}(0,0)%
\includegraphics{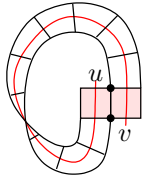}%
\end{picture}%
\setlength{\unitlength}{4144sp}%
\begingroup\makeatletter\ifx\SetFigFont\undefined%
\gdef\SetFigFont#1#2#3#4#5{%
  \reset@font\fontsize{#1}{#2pt}%
  \fontfamily{#3}\fontseries{#4}\fontshape{#5}%
  \selectfont}%
\fi\endgroup%
\begin{picture}(825,1049)(88,42)
\put(586,614){\makebox(0,0)[lb]{\smash{{\SetFigFont{10}{12.0}{\familydefault}{\mddefault}{\updefault}{\color[rgb]{0,0,0}$u$}%
}}}}
\put(766,254){\makebox(0,0)[lb]{\smash{{\SetFigFont{10}{12.0}{\familydefault}{\mddefault}{\updefault}{\color[rgb]{0,0,0}$v$}%
}}}}
\end{picture}%

  \caption{directly self-contact}
  \label{fig:ws}
\end{figure}

Let $H$ be a hyperplane in $G\backslash X$ and 
$N(H)$ the cubical neighborhood of $H$.
Let $C$ be a cube in $N(H)$ and $M$ the unique midplane of $C$
that belongs to $H$.
$M$ is unique because we are assuming that $H$ embeds in $G\backslash X$.
Let $u$ and $v$ be two vertices of $C$ not separated by $M$
such that $d(u,v) = d-1$, where $d$ is the dimension of $C$.
Then, we say that $C$ is spanned by $u$, $v$ and $H$.
If there is no hyperplanes that self-contacts,
then $C$ can be determined uniquely by $u$, $v$ and $H$.
(See Figure~\ref{fig:ws}).

\begin{remark}\label{rem:413}
By definition, if a cube complex is special in the sense of \cite{MR2377497},
then each hyperplane embeds, and it has no hyperplane of self-contact,
\end{remark}

Let $P_{s, t, g}$ be the set of pairs of letters such that 
it may appear in a word in $L({\mathcal M}_{s,t,g})$,
where ${\mathcal M}_{s, t, g}$ is the automaton with the same set of states
and transition as $\mathcal M$ but having the initial state $(s, t, g)$,
and accept states
\[
 \{(s',t',g') | \dim(g')=\dim(g)\}.
\]
In other words, $(x,y)$ is in $P_{s, t, g}$
if there exist a sequence $(x_0,y_0), \ldots, (x_{n-1}, y_{n-1}),$ $(x, y)$ 
which is accepted by ${\mathcal M}_{s, t, g}$.
(Note that $P_{s, t, g}$ depends only on the strongly connected component of $\mathcal M$.)

\begin{lemma}\label{lem:cor}
If each hyperplane in $G\backslash X$ is embedding and 
does not self-contact,
then, for each $(s, t, g)$ with $g\neq id$,
there exist two subsets $A', A''\subset A$ and
a one to one correspondence $f:A'\to A''$ such that
\[
  P_{s, t, g} = \{ (x,y)\in A'\times A'' | f(x)=y \}
\]
\end{lemma}

\begin{proof}
$A$ and $A'$ are determined by taking first and second projections of
$P_{s, t, g}$.

It suffices to show that if there are two sequences
$(x_1, y_1), \dots, (x_n, y_n)$ and
$(x_1', y_1'), \dots, (x_m', y_m')$
both accepted by ${\mathcal M}_{s,t,g}$,
then $x_n=x_m'$ implies $y_n=y_m'$, as well as
$y_n=y_m'$ implies $x_n=x_m'$.
In this case, we can define $f(x_n)=y_n$.
We prove the former case.

Define $(s_i, t_i, g_i)$ and $(s_j', t_j', g_j')$ inductively by
\begin{align*}
(s_i, t_i, g_i) & = \mu((s_{i-1}, t_{i-1}, g_{i-1}), (x_{i}, y_{i})), \\
(s_j', t_j', g_j') & = \mu((s_{i-j}', t_{i-j}', g_{i-j}'), (x_{j}', y_{j}'))
\end{align*}
for $i=1,\dots, n$ and $j=1,\dots,m$.
Let $v, u, v_i, u_i, v_j', u_j'$ be vertices in $G\backslash X$
corresponding to $s, t, s_i, t_i, s_j'$ and $t_j'$, respectively.
(For the correspondence between the set of states of the word acceptor
and $G\backslash X$, see the paragraph after Lemma~\ref{lem:as}.)
Put $d=\dim(s,t,g)$, and
denote the hyperplanes separating $v$ and $u$ by $H_1,\dots,H_d$.
(Recall that $S(v,u)=
S(u_i,u_i)=S(v_j',u_j')$ for all $i$ and $j$.)
Let $C_k$ be the cube 
spanned by $v_{n-1}$, $v_{n}$ and $H_k$ for $k=1,\dots,d$.
Recall that, by the assumption
(each hyperplane embeds and does not self-contact), 
each $C_k$ is uniquely determined by $v_{n-1}$, $v_{n}$ and $H_k$.
Let $C$ be the cube spanned by $v_{n-1}$, $v_{n}$, $u_{n-1}$, and $u_{n}$.
Then, $C$ is spanned by $C_1,\dots,C_d$ ($C$ contains $C_1,\dots,C_d$),
and it is uniquely determined by
$v_{n-1}, v_n$ and $S(v,u)$.
Define $C'$ in the same way.
It is uniquely determined by
$v_{n-1}', v_n'$ and $S(v,u)$.
If $x_n=x_n'$, then we have $v_{n-1} = v_{n-1}'$, $v_n=v_n'$ and $C=C'$.
Thus, $u_{n-1} = u_{n-1}'$ and $u_n=u_n'$.
Hence we have $y_n=y_m'$. 
\end{proof}

This is our main theorem in this section.

\begin{theorem}\label{thm:cube}
Let $G$ be a group acting effectively,
cellularly, properly discontinuously and cocompactly 
on a CAT(0) cube complex $X$.
If each hyperplane in $G\backslash X$ is embedding and
does not self-contact
and $G$ is not word hyperbolic,
then, $G$ contains $\ZZ+\ZZ$ subgroup.
\end{theorem}

\begin{proof}
First, note that Theorem~\ref{thm:track} works 
for the groupoid $\pi(G\backslash X)$.
By Lemma~\ref{lem:as},
Its Niblo--Reeves automatic structure described in the previous subsection
is prefix closed, weakly geodesic and has uniqueness property.
$X$ is not hyperbolic since $X$ and the Cayley graph of $G$ is quasi-isometric.
Hence, there exists an $n$-track of length $n$ in $X$ for any $n>0$.

For two vertices $\tilde v, \tilde v'$ in $X$ with $d(\tilde v, \tilde v')=1$,
we  denote by $\dim(\tilde v, \tilde v')$ the dimension of the cube spanned by $\tilde v$ and $\tilde v'$.
(Recall that we use normal cube-paths to define metric.)

Fix a vertex $\iota$ as the base point in $X$.
Let $T = \{ t_{1}, \ldots, t_{n} \}$ be an $n$-tracks of length $n$ in $X$.
We will improve $T$ in two steps.  

Step one.
The vertices in $T$ can be identified with $\{1, 2, \dots, n\} \times \{0, 1, \dots, n\}$,
where $(i,j)$ corresponds to $j$-th vertex in $t_i$,
which was denoted by $w_i(r+j)$ in the previous section.
Here, we denote this vertex by $v_{i,j}$.
We want to consider consecutive subtracks (a block) of $T$ of the form
$T' = \{i_0, i_0+1, \dots, i_0+m-1\} \times \{j_0, j_0+1, \dots, j_0+m\}$
such that
\begin{equation}\label{eqn:dim}
\dim(v_{i,j}, v_{i+1,j}) = \dim(v_{i,j+1}, v_{i+1,j+1})
\text{ for any }(i,j)\in T'.
\end{equation}
By Lemma~\ref{lem:mono}, for each $i\in\{1, \dots, n\}$, 
the number of indices $j\in\{0, \dots, n\}$ with
$\dim(v_{i,j}, v_{i+1,j}) \neq \dim(v_{i,j+1}, v_{i+1,j+1})$
is smaller than the maximal dimension of cubes in $G\backslash X$.
Thus, it is easy to see that for any $m>0$, 
there exists $n$ such that any $n$-tracks of length $n$ contains
an $m$-tracks of length $m$ that satisfies (\ref{eqn:dim}).
By abusing the notation, 
we refer to this subtracks by the same symbol $T= \{ t_{1}, \ldots, t_{n} \}$
and denote its size by $n$.
(Note that we do not assume that 
$\dim(v_{i,j}, v_{i+1,j}) = \dim(v_{i+1,j}, v_{i+2,j})$ in $T$.)
Step one is finished.

Step two.
Recall that, for two vertices $\tilde v, \tilde v'$ in $X$, 
we denote by $S(\tilde v,\tilde v')$ the set of hyperplanes separating them.
By the condition (\ref{eqn:dim}), for each $i\in\{1,\dots,n-1\}$,
there exists a set of hyperplanes 
${\mathfrak H}_i$ such that
$S(v_{i,j}, v_{i+1,j}) = {\mathfrak H}_i$ for any $j\in\{0,\dots,n\}$.
By Lemma~\ref{lem:distance}, for each hyperplane $H\in \cup_i {\mathfrak H}_i$,
we have $d(\iota, H) < d(\iota, v_{1,0})$.
%
Now, 
put $m = d(v_{1,n}, v_{n,n}) + 1$ and
let $C_1,\dots,C_{m-1}$ be the normal cube-path from $v_{1,n}$ to $v_{n,n}$.
We denote the vertices of this normal cube-path by
$v_{1}', v_{2}',\dots, v_{m}'$.
For $i\in\{1,\dots,m\}$,
let $t'_i$ be the postfix (tail) of normal cube-path from the base point $\iota$ to $v_i'$
with length $m$.
We claim that $T' = \{t_1',\dots,t_m'\}$ is a $m$-tracks of length $m$.
To see this, 
let $\mathfrak H_i'$ be the set of hyperplanes separating 
$v_{i}'$ and $v_{i+1}'$.
Then,
\[
   \bigcup_{i=1}^{m-1} {\mathfrak H}_i' = S(v_{1,n}, v_{n,n}) 
   \subset \bigcup_{i=1}^{n-1} {\mathfrak H}_i.
\]
Thus, we have 
\begin{equation}\label{eqn:d}
d(\iota, H') < d(\iota, v_{1,0})
\end{equation}
for each hyperplane $H' \in \cup_i {\mathfrak H}_i'$.
Thus, by Lemma~\ref{lem:distance}, 
we have $t_i' \cap t_j' = \emptyset$ if $i\neq j$, and
$d(\iota, v_i') = d(\iota, v_{1,n})$ for $i\in\{1,\dots,m\}$,
and $T'$ is a $m$-tracks of length $m$.
Moreover, (\ref{eqn:d}) implies that $T'$ satisfies the condition (\ref{eqn:dim}).
(And, needles to say, $v_{1}',\dots,v_{m}'$ was the vertices of a normal cube-path.)
The size $m$ can be smaller than $n$.
But, since $X$ is locally finite, if $n$ was large enough,
we may assume that $m=d(v_{1,n},v_{n,n})-1$ is as large as we want.
By abusing notation, 
we refer to this new tracks by the same symbol $T$ and 
denote its size by $n$, and vertices by $v_{i,j}$.
Step two is finished.

We denote by $t_{i,j} (\in A)$ the label of the directed edge 
from $v_{i,j}$ to $v_{i, j+1}$.

Fix $i\in\{1,\dots,n\}$, and
let us consider a pair of consecutive tracks $t_i$ and $t_{i+1}$ in $T$.
By definition,
for each pair of vertices $(v_{i,j}, v_{i+1,j})$ on $t_i$ and $t_{i+1}$,
there is a corresponding state $(s_j,t_j,g_j)\in\mathcal M$.
(It is unique by the uniqueness property of the automatic structure.)
Since $\dim(v_{i,j}, v_{i+1,j})$ is constant for all $j$,
each state $(s_j,t_j,g_j)$ can be regarded as a state 
in ${\mathcal M}_{s_0, t_0, g_0}$,
and clearly, $g_0\neq id$.
(For the definition of ${\mathcal M}_{s_0, t_0, g_0}$,
see paragraphs after Remark~\ref{rem:413}.)
By Lemma~\ref{lem:cor}, 
there exists a one to one map $f_i:A\to A$ 
such that $t_{i+1, j} = f_i(t_{i, j})$ for any $j$.
(In Lemma~\ref{lem:cor},  the map was from a subset $A'$ of $A$ to
another subset $A''$, but it is easy to extend this map from $A$ to $A$.
The extension is not unique but the complement of $A'$ will not be used anyway.)
By combining these maps,
for each $i, i'\in\{1,\dots,n\}$,
we have a one to one map
$f_{i,i'}:A \to A$ such that $t_{i', j} = f_{i,i'}(t_{i, j})$ for any $j$.

For vertices $\tilde v, \tilde v'\in X$, 
we denote the word in $A^*$ 
given by the normal cube-path from $\tilde v$ to $\tilde v'$ by $[\tilde v,\tilde v']$.

Let $D$ be the maximal dimension of the cubes in $X$.
We claim that,
if $n$ (the size of tracks) was large enough, 
$T$ contains a set of indices $I=\{i_0, i_1, \dots, i_D\}$ and
$J=\{j_0, j_1\}$ which satisfies the following conditions:
\begin{itemize}
\item[(C1)] $f_{i,i'}:A \to A$ is the identity map for each $i, i'\in I$.
\item[(C2)] $v_{i_0,j_0}$ and $v_{i_0,j_1}$ correspond 
      to the same state in the word acceptor,
\item[(C3)] $v_{i,j}$ projects to the same vertex, say $v$, 
      in $G\backslash X$ for all $i\in I, j\in J$.
\item[(C4)] $[v_{i,j_0}, v_{i',j_0}] =  [v_{i,j_1}, v_{i',j_1}]$ 
      for all $i,i'\in I$.
      Moreover, there exists a letter $\alpha\in A$ such that
      $[v_{i,j_0}, v_{i',j_0}]$ begins with $\alpha$
      for any $i<i'$.
\end{itemize}
To see the claim observe that
\begin{enumerate}
\item the number of permutations $A \to A$ is finite
\item the number of states in the word acceptor is finite
\item the number of vertices in $G\backslash X$ is finite
\item the word lengths of $[v_{i,j}, v_{i',j}]$ is less than $|i-i'|$ 
      for any $j$,
\end{enumerate}
and these numbers does not depend on the choice of $T$.
Then, it is easy to calculate the value of $n$ needed so that $T$
contains $I$ and $J$ which satisfy
the above conditions.
(At the moment, 
$D$ does not have to be he maximal dimension of the cubes in $X$.)
From now on, we assume that our $T$ is large enough so that 
it satisfies these conditions.

\begin{figure}[htbp]
  \centering

\begin{picture}(0,0)%
\includegraphics{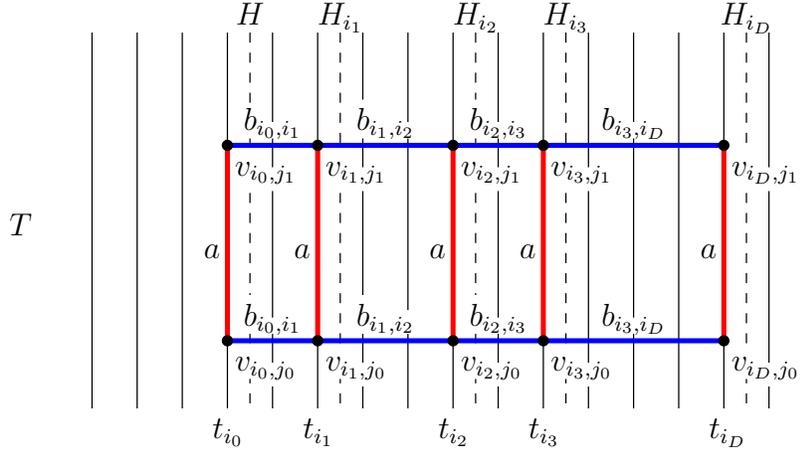}%
\end{picture}%
\setlength{\unitlength}{4144sp}%
\begingroup\makeatletter\ifx\SetFigFont\undefined%
\gdef\SetFigFont#1#2#3#4#5{%
  \reset@font\fontsize{#1}{#2pt}%
  \fontfamily{#3}\fontseries{#4}\fontshape{#5}%
  \selectfont}%
\fi\endgroup%
\begin{picture}(4707,2703)(-644,-1660)
\put(496,-511){\makebox(0,0)[rb]{\smash{{\SetFigFont{12}{14.4}{\familydefault}{\mddefault}{\updefault}{\color[rgb]{0,0,0}$a$}%
}}}}
\put(1036,-511){\makebox(0,0)[rb]{\smash{{\SetFigFont{12}{14.4}{\familydefault}{\mddefault}{\updefault}{\color[rgb]{0,0,0}$a$}%
}}}}
\put(1846,-511){\makebox(0,0)[rb]{\smash{{\SetFigFont{12}{14.4}{\familydefault}{\mddefault}{\updefault}{\color[rgb]{0,0,0}$a$}%
}}}}
\put(2386,-511){\makebox(0,0)[rb]{\smash{{\SetFigFont{12}{14.4}{\familydefault}{\mddefault}{\updefault}{\color[rgb]{0,0,0}$a$}%
}}}}
\put(3466,-511){\makebox(0,0)[rb]{\smash{{\SetFigFont{12}{14.4}{\familydefault}{\mddefault}{\updefault}{\color[rgb]{0,0,0}$a$}%
}}}}
\put(451,-1591){\makebox(0,0)[lb]{\smash{{\SetFigFont{12}{14.4}{\familydefault}{\mddefault}{\updefault}{\color[rgb]{0,0,0}$t_{i_0}$}%
}}}}
\put(1801,-1591){\makebox(0,0)[lb]{\smash{{\SetFigFont{12}{14.4}{\familydefault}{\mddefault}{\updefault}{\color[rgb]{0,0,0}$t_{i_2}$}%
}}}}
\put(2341,-1591){\makebox(0,0)[lb]{\smash{{\SetFigFont{12}{14.4}{\familydefault}{\mddefault}{\updefault}{\color[rgb]{0,0,0}$t_{i_3}$}%
}}}}
\put(3421,-1591){\makebox(0,0)[lb]{\smash{{\SetFigFont{12}{14.4}{\familydefault}{\mddefault}{\updefault}{\color[rgb]{0,0,0}$t_{i_D}$}%
}}}}
\put(991,-1591){\makebox(0,0)[lb]{\smash{{\SetFigFont{12}{14.4}{\familydefault}{\mddefault}{\updefault}{\color[rgb]{0,0,0}$t_{i_1}$}%
}}}}
\put(676,884){\makebox(0,0)[b]{\smash{{\SetFigFont{12}{14.4}{\familydefault}{\mddefault}{\updefault}{\color[rgb]{0,0,0}$H$}%
}}}}
\put(1216,884){\makebox(0,0)[b]{\smash{{\SetFigFont{12}{14.4}{\familydefault}{\mddefault}{\updefault}{\color[rgb]{0,0,0}$H_{i_1}$}%
}}}}
\put(2026,884){\makebox(0,0)[b]{\smash{{\SetFigFont{12}{14.4}{\familydefault}{\mddefault}{\updefault}{\color[rgb]{0,0,0}$H_{i_2}$}%
}}}}
\put(2566,884){\makebox(0,0)[b]{\smash{{\SetFigFont{12}{14.4}{\familydefault}{\mddefault}{\updefault}{\color[rgb]{0,0,0}$H_{i_3}$}%
}}}}
\put(3646,884){\makebox(0,0)[b]{\smash{{\SetFigFont{12}{14.4}{\familydefault}{\mddefault}{\updefault}{\color[rgb]{0,0,0}$H_{i_D}$}%
}}}}
\put(-629,-376){\makebox(0,0)[rb]{\smash{{\SetFigFont{12}{14.4}{\familydefault}{\mddefault}{\updefault}{\color[rgb]{0,0,0}$T$}%
}}}}
\put(3556,-1186){\makebox(0,0)[lb]{\smash{{\SetFigFont{12}{14.4}{\familydefault}{\mddefault}{\updefault}{\color[rgb]{0,0,0}$v_{i_D, j_0}$}%
}}}}
\put(1936,-1186){\makebox(0,0)[lb]{\smash{{\SetFigFont{12}{14.4}{\familydefault}{\mddefault}{\updefault}{\color[rgb]{0,0,0}$v_{i_2, j_0}$}%
}}}}
\put(1126,-1186){\makebox(0,0)[lb]{\smash{{\SetFigFont{12}{14.4}{\familydefault}{\mddefault}{\updefault}{\color[rgb]{0,0,0}$v_{i_1, j_0}$}%
}}}}
\put(586,-1186){\makebox(0,0)[lb]{\smash{{\SetFigFont{12}{14.4}{\familydefault}{\mddefault}{\updefault}{\color[rgb]{0,0,0}$v_{i_0, j_0}$}%
}}}}
\put(1126,-16){\makebox(0,0)[lb]{\smash{{\SetFigFont{12}{14.4}{\familydefault}{\mddefault}{\updefault}{\color[rgb]{0,0,0}$v_{i_1, j_1}$}%
}}}}
\put(1936,-16){\makebox(0,0)[lb]{\smash{{\SetFigFont{12}{14.4}{\familydefault}{\mddefault}{\updefault}{\color[rgb]{0,0,0}$v_{i_2, j_1}$}%
}}}}
\put(586,-16){\makebox(0,0)[lb]{\smash{{\SetFigFont{12}{14.4}{\familydefault}{\mddefault}{\updefault}{\color[rgb]{0,0,0}$v_{i_0, j_1}$}%
}}}}
\put(2476,-16){\makebox(0,0)[lb]{\smash{{\SetFigFont{12}{14.4}{\familydefault}{\mddefault}{\updefault}{\color[rgb]{0,0,0}$v_{i_3, j_1}$}%
}}}}
\put(3556,-16){\makebox(0,0)[lb]{\smash{{\SetFigFont{12}{14.4}{\familydefault}{\mddefault}{\updefault}{\color[rgb]{0,0,0}$v_{i_D, j_1}$}%
}}}}
\put(1486,254){\makebox(0,0)[b]{\smash{{\SetFigFont{12}{14.4}{\familydefault}{\mddefault}{\updefault}{\color[rgb]{0,0,0}$b_{i_1, i_2}$}%
}}}}
\put(2161,254){\makebox(0,0)[b]{\smash{{\SetFigFont{12}{14.4}{\familydefault}{\mddefault}{\updefault}{\color[rgb]{0,0,0}$b_{i_2, i_3}$}%
}}}}
\put(2971,254){\makebox(0,0)[b]{\smash{{\SetFigFont{12}{14.4}{\familydefault}{\mddefault}{\updefault}{\color[rgb]{0,0,0}$b_{i_3, i_D}$}%
}}}}
\put(811,254){\makebox(0,0)[b]{\smash{{\SetFigFont{12}{14.4}{\familydefault}{\mddefault}{\updefault}{\color[rgb]{0,0,0}$b_{i_0, i_1}$}%
}}}}
\put(811,-916){\makebox(0,0)[b]{\smash{{\SetFigFont{12}{14.4}{\familydefault}{\mddefault}{\updefault}{\color[rgb]{0,0,0}$b_{i_0, i_1}$}%
}}}}
\put(1486,-916){\makebox(0,0)[b]{\smash{{\SetFigFont{12}{14.4}{\familydefault}{\mddefault}{\updefault}{\color[rgb]{0,0,0}$b_{i_1, i_2}$}%
}}}}
\put(2161,-916){\makebox(0,0)[b]{\smash{{\SetFigFont{12}{14.4}{\familydefault}{\mddefault}{\updefault}{\color[rgb]{0,0,0}$b_{i_2, i_3}$}%
}}}}
\put(2971,-916){\makebox(0,0)[b]{\smash{{\SetFigFont{12}{14.4}{\familydefault}{\mddefault}{\updefault}{\color[rgb]{0,0,0}$b_{i_3, i_D}$}%
}}}}
\put(2476,-1186){\makebox(0,0)[lb]{\smash{{\SetFigFont{12}{14.4}{\familydefault}{\mddefault}{\updefault}{\color[rgb]{0,0,0}$v_{i_3, j_0}$}%
}}}}
\end{picture}%

  \caption{Finding $\ZZ+\ZZ$ subgroup: Vertical solid lines are tracks of $T$.  Dashed lines are hyperplanes.  We omit $b_{i_0, i_2}, b_{i_0, i_3},$ etc.
to simplify the picture.}
  \label{fig:t}
\end{figure}

Now, define $a = [v_{i_0,j_0}, v_{i_0,j_1}]$ 
and $b_{i,i'} = [v_{i,j_0}, v_{i',j_0}]$
($= [v_{i,j_1}, v_{i',j_1}]$)  for $i,i'\in I$.
We consider that these elements are in $\pi_1(G\backslash X, v) \simeq G$,
because of (C3).
Note that these elements connect vertices $v_{i,j}$ ($i\in I, j\in J$) 
``vertically'' and ``horizontally.''
By (C1) and (C3), we have $ab_{i,j} = b_{i,j}a$ for any $i\in I$, $j\in J$.
By (C2), $a^n$ (with possibly some prefix) is accepted by the
word acceptor for any $n>0$.
Since the automatic structure is weakly geodesic and 
has the uniqueness property.
$a$ is torsion free and $a^n$ is a normal cube-path for any $n>0$.

We claim that, at least one element of $\{b_{i,i'}\}_{i,i'\in I}$ is torsion free.
Let $H$ be a hyperplane separating $v_{i_0,j_0}$ and $v_{i_0+1,j_0}$.
(By (C4), $v_{i_0,j_0}$ and $v_{i_0+1,j_0}$ span a (directed) cube labeled $\alpha$.)
We want to consider the action
of elements in $\pi_1(G\backslash X)$ on hyperplanes.
For the sake of simplicity,
(after conjugation,) suppose that the vertex $v_{i_0,j_0}$ is the
base point of $X$ as Cayley graph of $\pi_1(G\backslash X)$.
Define $H_i$ as the image of $H$ by the action of $b_{i_0,i}$ for each $i\in I$.
By the second condition in (C4), $H_i$ 
intersects a cube with label $\alpha$ and separates $t_i$ and $t_{i+1}$.
It follows that $H_i \neq H_{i'}$ for any $i\neq i'$,
because $v_{1,n}, v_{2,n}, \dots, v_{n,n}$ are vertices of a normal cube-path
(after step two)
and a hyperplane separates a normal cube-path at most once
(Remark~\ref{rem:NR3}).

Since a family of pairwise intersecting hyperplanes 
have a common point of intersection, 
the cardinality of such a family is bounded by the maximal dimension $D$ of $X$.
(See Theorem 4.14 in \cite{MR1347406}.)
Thus, in $I=\{i_0, i_1, \dots, i_D\}$,  there exist indices $i, i'$ with $i<i'$ 
such that $H_i$ and $H_{i'}$ do not intersect,

We claim that $b_{i,i'}$ is torsion free.
For the sake of simplicity,
(after conjugation,) suppose that the vertex $v_{i,j_0}$ is the
base point of $X$ as Cayley graph of $\pi_1(G\backslash X)$.
Then, $b_{i,i'}(H_i) = H_{i'}$,
and we have $H_i \cap b_{i,i'}(H_i) = \emptyset$.
Let ${\mathcal H}$ be a halfspace of $H_i$ containing $v_{n,n}$.
By the second condition in (C4),
$b_{i,i'}({\mathcal H})$ also contains $v_{n,n}$.
Thus, $b_{i,i'}({\mathcal H_i}) \subset {\mathcal H_i}$,
Then, the orbit of $H_i$ under any positive power of $b_{i,i'}$ 
is contained in ${\mathcal H_i}$.
Hence $b_{i,i'}$ is torsion free.
Therefore, at least one element of $\{b_{i,i'}\}_{i,i'\in I}$ is torsion free
and we denote this element by $b$.

Next, we claim that $\langle a, b \rangle$ is rank two.
Let $i$ be the index used to define $b=b_{i,i'}$.
Recall that positive powers of $b$ gives 
a nested sequence of halfspaces 
${\mathcal H}_i \supset b({\mathcal H_i}) \supset b^2 ({\mathcal H_i})
\supset \cdots$.
If $\langle a, b \rangle$ is cyclic,
then there exist $p,q$ such that $a^p = b^q$.
Since $b^q \in b^{q-1} ({\mathcal H_i})$,
we have $a^p \in b^{q-1} ({\mathcal H_i}) \subset {\mathcal H}_i$.
Recall that $a^n$ is a normal cube-path for any $n>0$.
Recall also that it stayed outside of ${\mathcal H_i}$ when it was in $[v_{i,j_0}, v_{i,j_1}]$. 
But, $St(v_{i,j_0}) \cap H_i\neq\emptyset$
and this is a contradiction.
(Recall Remark~\ref{rem:NR3}.)
Thus, this subgroup is not cyclic, but is of rank two.

Hence $\langle a, b \rangle$ is $\ZZ+\ZZ$ subgroup, and the theorem is proved.
\end{proof}

Finally, we ask questions that we hope interesting.

\begin{question}
Is the condition "no hyperplane of direct self-contact" necessary?
\end{question}

\begin{question}
More systematic method is known to show that
a group acting on a space is automatic.
See \cite{MR2239447}.
Can one generalize the above result for this setting?
\end{question}


\end{document}